\documentclass[a4paper]{amsart}

\usepackage{framed}
\usepackage{amsmath}
\usepackage{amssymb}
\usepackage{amsmath}
\usepackage{mathrsfs}
\usepackage{amsthm}
\usepackage[all]{xy}
\usepackage{color}
\usepackage[top=23truemm,bottom=23truemm,left=28truemm,right=28truemm]{geometry}

\newcommand{\ra}{\rightarrow}

\newcommand{\hra}{\hookrightarrow}
\newcommand{\thra}{\twoheadrightarrow}
\newcommand{\lra}{\longrightarrow}

\newcommand{\id}{\mathop{\mathrm{id}}\nolimits}

\newcommand{\sepK}{K^{\rm sep}}
\newcommand{\rk}{\mathop{\mathrm{rk}}}
\newcommand{\I}{\cI_\phi}
\newcommand{\Ip}{\cI_{\phi,\frp}}
\newcommand{\T}{\cT(\pi_\frp\langle\phi\rangle)}
\newcommand{\Lp}{\pi_\frp\langle \phi \rangle}

\newcommand{\Ker}{\mathop{\mathrm{Ker}}}

\providecommand{\Hom}[3]{{\rm Hom^{ }}_{#1}(#2,#3)}

\providecommand{\End}[2]{{\rm End^{ }}_{#1}(#2)}

\providecommand{\Gal}[2]{{\rm Gal}(#1/#2)}

\providecommand{\T}[1]{\tilde{#1}}

\providecommand{\ad}[1]{#1\{ \tau \}}

\makeatletter 
\@tempcnta\z@
\loop\ifnum\@tempcnta<26
\advance\@tempcnta\@ne
\expandafter\edef\csname rm\@Alph\@tempcnta\endcsname{\noexpand\mathrm{\@Alph\@tempcnta}}
\expandafter\edef\csname s\@Alph\@tempcnta\endcsname{\noexpand\mathscr{\@Alph\@tempcnta}}
\expandafter\edef\csname b\@Alph\@tempcnta\endcsname{\noexpand\mathbb{\@Alph\@tempcnta}}
\expandafter\edef\csname c\@Alph\@tempcnta\endcsname{\noexpand\mathcal{\@Alph\@tempcnta}}
\expandafter\edef\csname rm\@alph\@tempcnta\endcsname{\noexpand\mathrm{\@alph\@tempcnta}}
\expandafter\edef\csname fr\@alph\@tempcnta\endcsname{\noexpand\mathfrak{\@alph\@tempcnta}}
\expandafter\edef\csname fr\@Alph\@tempcnta\endcsname{\noexpand\mathfrak{\@Alph\@tempcnta}}
\repeat

\theoremstyle{breeak}	
\newtheorem{thm}{Theorem}[section]
\newtheorem{prop}[thm]{Proposition}
\newtheorem{lem}[thm]{Lemma}

\newtheorem{cor}[thm]{Corollary}

\theoremstyle{definition}
\newtheorem{definition}[thm]{Definition}
\newtheorem{rem}[thm]{Remark}
\newtheorem{ex}[thm]{Example}

 \makeatletter
    \@addtoreset{equation}{section}
  \makeatother

\title{Parametrization of virtually $K$-rational Drinfeld modules of rank two}
\date{}
\author{Yoshiaki Okumura}
\begin{document}
\maketitle
\thispagestyle{empty}


\begin{abstract}
For an extension $K/\mathbb{F}_q(T)$ of the rational function field over a finite field, we introduce the notion of virtually $K$-rational Drinfeld modules as a function field analogue of $\mathbb{Q}$-curves.
Our goal in this article is to prove that all virtually $K$-rational  Drinfeld modules of rank two with no complex multiplication are parametrized up to isogeny by  $K$-rational points of a quotient curve  of the  Drinfeld modular curve $Y_0(\mathfrak{n})$ with some square-free level $\mathfrak{n}$.
This is an analogue of Elkies' well-known result on $\mathbb{Q}$-curves.
\end{abstract}


\pagestyle{myheadings}
\markboth{Y.\ Okumura}{Parametrization of virtually $K$-rational Drinfeld modules of rank two}

\renewcommand{\thefootnote}{\fnsymbol{footnote}}
\renewcommand{\thefootnote}{\arabic{footnote}}
\renewcommand{\thefootnote}{\fnsymbol{footnote}}
\renewcommand{\thefootnote}{\arabic{footnote}}


\section{Introduction}
 An elliptic curve  $E$ over an algebraic closure $\bar{\bQ}$ of the rational number field $\bQ$ is called  a {\it $\bQ$-curve} if $E$  is isogenous to the  conjugate ${}^sE$ for any  $s \in G_\bQ=\Gal{\bar{\bQ}}{\bQ}$.
For the first time, the notion of  $\bQ$-curves was   introduced by Gross \cite{Gross}  with a much more restrictive definition, and later considered  by 
 Ribet in a general setting.
There are  various interesting arithmetic aspects of $\bQ$-curves; 
for instance,  Ribet showed in \cite{Rib04} that all non-CM  $\bQ$-curves are  quotients of some  abelian varieties of ``{\it $\mathrm{GL}_2$-type}''  and every non-CM $\bQ$-curve is modular, meaning that it is a quotient of the modular Jacobian $J_1(N)_{\bar{\bQ}}$ for some $N$.

The condition of being  $\bQ$-curves  is invariant by isogeny and  so  the
classification of isogeny classes of $\bQ$-curves is a natural problem.
As an answer to this,  Elkies proved in \cite{Ep} that every non-CM $\bQ$-curve  is isogenous to that whose $j$-invariant  arising from a $\bQ$-rational point of $Y_*(N)$, the  quotient  of the elliptic modular curve $Y_0(N)$ of some square-free level $N$ by all Atkin-Lehner involutions.
(Notice that \cite{Ep} is unpublished but its
 revised version   \cite{E}  is available.
It is worth pointing out  that Elkies in fact considered ``$k$-curves'' for  an arbitrary number field $k$.)
For any square-free $N$,
Elkies also showed  that  
any $\bQ$-curve $E$ arising from  a point of $Y_*(N)(\bQ)$ can be defined over a  {\it  polyquadratic extension} of $\bQ$ (i.e., a finite abelian extension of $\bQ$ with  Galois group $G \cong (\bZ/2\bZ)^n$) and  that $E$ admits an isogeny to ${}^sE$ of degree dividing $N$ for any $s \in G_\bQ$.     
Such $E$ is called a {\it central} $\bQ$-curve of degree $N$.
Therefore the existence of non-CM $\bQ$-rational points of $Y_*(N)$ is equivalent to that of non-CM central $\bQ$-curves of degree $N$.
Elkies  conjectured that there are no non-CM $\bQ$-rational points of $Y_*(N)$ if $N$ is sufficiently large.

The purpose of this article is to prove a function field analogue of Elkies' result.
As is well-known,  there are many beautiful analogies between number fields and  function fields.
In \cite{Dri}, Drinfeld introduced  the analogue of elliptic curves under the name {\it elliptic modules}, which are today called {\it Drinfeld modules}.
Drinfeld modules share many arithmetic properties with  elliptic curves and so we may expect that there is a rich theory of a Drinfeld module  analogue of $\bQ$-curves.
Let $A=\bF_q[T]$ be the polynomial ring over a  fixed finite field $\bF_q$ of characteristic $p$ and let $Q=\bF_q(T)$  be the rational function field. 
Let $K$ be a field with $Q \subset K$.
Based on  the analogy, we introduce the notion of  {\it virtually $K$-rational} Drinfeld $A$-modules; see Definition \ref{defKvirtual}.
Let  $\frn \subset A$ be a non-zero ideal.
For $p \neq 2$,  it follows by  Proposition \ref{rat} that  rank-two Drinfeld $A$-modules arising from $K$-rational points of $Y_*(\frn)$  are  virtually $K$-rational.
Here $Y_*(\frn)$ is  the quotient of the Drinfeld modular curve $Y_0(\frn)$ by all Atkin-Lehner involutions.
Even if $p=2$, at least  all non-CM Drinfeld $A$-modules arising from $Y_*(\frn)(K)$ are  virtually $K$-rational.
Adapting  Elkies' graph-theoretic method  to Drinfeld $A$-modules, we obtain the main result:

 \begin{thm}\label{thmmain}
Let $\phi$ be a non-CM virtually $K$-rational Drinfeld $A$-module of rank two.
Then there exists a non-zero square-free ideal $\frn \subset A$, depending only on the isogeny class of $\phi$, such that
\begin{itemize}
\item[(i)] $\phi$ is isogenous to a Drinfeld $A$-module arising from a $K$-rational point of $Y_*(\frn)$,
\item[(ii)] if an ideal $\frn'$ satisfies {\rm (i)}, then $\frn \mid \frn'$.
\end{itemize} 
\end{thm}

This theorem gives an analogue of Elkies' classification of non-CM $\bQ$-curves.
Thus it allows us to relate a Diophantine problem on $Y_*(\frn)(K)$ with the existence of virtually $K$-rational Drinfeld $A$-modules; see also  Remark \ref{rem:central}.

The organization of this article is as follows.
In Section 2, we  review well-known  facts on Drinfeld $A$-modules and recall the definitions of     {\it degree} and  {\it dual} of isogenies, which are fundamental tools in our work. 
In Section 3, after showing basic properties of   Galois conjugates of Drinfeld $A$-modules, we define virtually $K$-rational Drinfeld $A$-modules of arbitrary rank.  
We also give a non-trivial example of virtually $Q$-rational Drinfeld $A$-modules  of rank two in Example \ref{exKvirtual}.
 In the remaining sections, 
we restrict our  attention  to  the rank-two and non-CM case.
Section 4 is devoted to a study of rational points of the curve $Y_*(\frn)$.
Using   moduli interpretation, we prove that if a $K$-rational point of $Y_*(\frn)$ satisfies some mild condition, then it gives raise to a family of   virtually $K$-rational  Drinfeld $A$-modules isogenous to each other; see Proposition \ref{rat}.
Finally, in Section 5, we give a proof of   Theorem \ref{thmmain}.
To find the  $\frn$ attached to a given virtually $K$-rational $\phi$  in Theorem \ref{thmmain}, we consider   a Galois action on an undirected tree (so called {\it isogeny tree}) associated with $\phi$.

\section{Isogenies of Drinfeld $A$-modules}

We  begin by fixing the notation.
As in Section 1, 
let $A=\bF_q[T]$ be the polynomial ring over the finite field $\bF_q$ with $q$-elements of characteristic $p>0$ and set  $Q=\bF_q(T)$.
Let $K$ be a field with $K \subset Q$ and fix an algebraic closure $\bar{K}$ of $K$.
Denote by $\sepK \subset \bar{K}$ the separable closure of $K$ and write $G_K:=\Gal{\sepK}{K}$ for its absolute Galois group.
Denote by $K\{ \tau \}$  the skew polynomial ring over $K$ in one variable $\tau$ satisfying $\tau c=c^q\tau$ for any $c \in K$.
It  is isomorphic to the ring $\End{\bF_q\mbox{\tiny -lin}}{\bG_{a,K}}$ of all $\bF_q$-linear endomorphisms of the additive group $\bG_{a,K}$ over $K$.
Define the {\it differential map}
\[
\partial:\ad{K} \ra K
\]
by $\partial(\sum_{i=0}^nc_i\tau^i)=c_0$.
It is an $\bF_q$-algebra homomorphism.

A {\it Drinfeld $A$-module}  over $K$ is an $\bF_q$-algebra homomorphism 
\[
\begin{array}{cccc}
\phi: & A & \ra & \ad{K} \\
 & a & \mapsto &  \phi_a
\end{array}
\]
 such that $\partial(\phi_a)=a$ for any $a \in A$ and $\phi_a \neq a $ for some $a \in A$.
By definition,  $\phi$   is  completely determined by  the image of $T$: $\phi_T=T + c_1\tau + \cdots + c_r\tau^r \in \ad{K}$ with $c_r \neq 0$.
The integer $r$ is called the {\it rank} of $\phi$ and denoted by $\rk \phi$.

\begin{rem}\label{remgeneralD}
Drinfeld modules are defined more generally: let $C$ be a smooth projective, geometrically irreducible curve over $\bF_q$ and let $\infty \in C$ be a fixed closed point.
Let $\cA$ be the ring of rational functions on $C$ regular outside $\infty$.
Let $\cF$ be a field equipped with a (not necessarily injective) $\bF_q$-algebra homomorphism $\iota:\cA \ra \cF$.
Then    a {\it Drinfeld $\cA$-module} over $\cF$ is an $\bF_q$-algebra homomorphism
\[
\phi:\cA \ra \ad{\cF}
\] 
such that  $\partial(\phi_a)=\iota(a)$ for any $a \in \cA$ and $\phi_a \neq \iota(a)$ for some $a \in \cA$.
In this paper, we mainly consider the case where $C=\bP^1_{\bF_q}$, $\infty=(1/T)$, $\cA=A=\bF_q[T]$, and $\cF=K$ with the inclusion $A \hra Q \subset K$.
\end{rem}

Let $\phi$ be a Drinfeld $A$-module over $K$ and let  $L$ be a  $K$-algebra.
For any $\lambda \in L$ and $\mu=\sum_{i=0}^{n}c_i\tau^i \in \ad{K}$, set
$\mu(\lambda):=\sum_{i=0}^{n}c_i\lambda^{q^i} \in L$.
Then $\phi$   endows the additive group $\bG_{a,K}(L)=L$ with a new $A$-module structure  by 
$a\cdot\lambda:=\phi_a(\lambda)$ for any $a \in A$ and $\lambda \in L$.
Denote  this  new $A$-module by ${}_\phi L$.
If $L=\bar{K}$, then 
for a non-zero $a \in A$, we denote the $A$-module of $a$-torsion points of $\phi$ by 
$
\phi[a]:=\left\{ \lambda \in {}_\phi\bar{K}; \phi_a(\lambda)=0 \right\}. 
$
It is in fact contained in   ${}_\phi\sepK$ because $\partial(\phi_a)=a \neq 0$.
Hence $G_K$ acts on $\phi[a]$.
For a non-zero ideal $\frn \subset A$, we also set
\[
\phi[\frn]:=\bigcap_{0 \neq a \in \frn} \phi[a].
\]
Clearly $\phi[\frn]=\phi[a]$  if  $\frn=(a)$.
It is known that $\phi[\frn]$ is a  $G_K$-stable free $A/\frn$-module of rank equal to $\rk \phi$.

Let $\phi$ and $\psi$ be Drinfeld $A$-modules over $K$.
An {\it isogeny} $\mu:\phi \ra \psi$  is a non-zero  element 
$\mu \in \ad{\bar{K}}$ such that 
\[
\mu\phi_a=\psi_a\mu
\]
 for any $a \in A$, and then $\phi$ is said to be {\it isogenous} to $\psi$.
Equivalently, $\mu$ is an isogeny if $\mu\phi_T^{}=\psi_T^{}\mu$ holds.
If $\mu \in \bar{K}^\times$, then  it is called  an {\it isomorphism}.
By definition, an isogeny $\mu:\phi \ra \psi$ induces a surjective $A$-module homomorphism
\[
\mu: {}_\phi\bar{K} \ra {}_\psi\bar{K}.
\]
Its kernel   $\Ker\mu$ is actually contained in ${}_\phi\sepK$ because $\mu$  satisfies $\partial(\mu) \neq 0$ by \cite[Chapter 1, $\S 4$]{DH87}.
Notice that  every isogeny between Drinfeld $A$-modules over $K$ is in fact contained in $\ad{\sepK}$ by \cite[Proposition 4.7.4.]{Goss}.

For any field $L/K$, if an isogeny  $\mu:\phi \ra \psi$ satisfies $\mu \in \ad{L}$, then it  is called an  
{\it $L$-rational isogeny} (or {\it $L$-isogeny}).
Two isogenies $\mu:\phi\ra\psi$ and $\eta:\phi' \ra \psi'$ are said to be  {\it $L$-equivalent} if there exist
$L$-rational isomorphisms $\nu:\phi \ra \phi'$ and $\lambda:\psi\ra\psi'$ such that $\eta=\lambda\mu\nu^{-1}$.
We write 
$
\Hom{L}{\phi}{\psi}:=\{\mu \in \ad{L};\ \mu\phi_T^{}=\psi_T^{}\mu \}
$ 
and $\End{L}{\phi}:=\Hom{L}{\phi}{\phi}$.
Then  $\phi_a \in \End{L}{\phi}$ for any $a \in A$.
Hence
 $
A
$ is a subring of   $\End{L}{\phi}$ and  $\Hom{L}{\phi}{\psi}$ becomes  an $A$-module by
$a\cdot\mu:= \mu\phi_a$.
It is known that 
 $\End{L}{\phi}$ is a commutative $A$-algebra and 
 a  free $A$-module of  rank $\leq \rk\phi$.
Therefore  $\End{L}{\phi} {\otimes_A}Q$ is a finite extension of $Q$.

The restriction of the differential map $\partial:\ad{K}\ra K$ to 
$\End{K}{\phi}$ induces an injective $A$-algebra homomorphism 
\[
\partial:\End{K}{\phi} \hra K,
\]
so that its  image  is an $A$-order of a finite extension $F/Q$ satisfying  $F \subset K$.
If $\rk\phi>1$ and $[F:Q]>1$, then we say that $\phi$ has {\it complex multiplication} (or {\it CM}) by $F$.
If in addition  $[F:Q]=\rk\phi$, then we especially say that $\phi$ has {\it full complex multiplication}.

\begin{rem}
Suppose that $\phi$ has CM by a finite extension $F/Q$.
Then $F$ has  only one place lying above the place 
$\infty=(1/T)$ of $Q$; see \cite[Proposition 4.7.17]{Goss} for instance.
If furthermore $\partial(\End{K}{\phi})$ is the ring of integers $\cO_F$ of $F$, then the inclusion $\End{K}{\phi} \hra \ad{K}$ determines a Drinfeld $\cO_F$-module $\Phi: \cO_F \ra \ad{K}$ such that 
 $\rk\Phi=\frac{\rk \phi}{[F:Q]}$ and $\Phi|_A=\phi$.
\end{rem}

Let $\mu:\phi \ra \psi$ be an isogeny between Drinfeld $A$-modules over $K$.
Since $\Ker\mu$ is a finite torsion $A$-submodule of ${}_\phi\sepK$, it  is isomorphic to $\bigoplus_{i=1}^nA/\frn_i$ for suitable $n$ and non-trivial  ideals $ \frn_i \subset A$.
Then  
the product 
\[
\deg \mu:=\prod_{i=1}^n \frn_i
\]
 is called the  {\it degree} of  $\mu$. 
The ideals $\frn_i$ are uniquely determined by $\mu$ and so is $\deg\mu$.

\begin{lem}\label{lemmulti}
Let $\mu:\phi_1 \ra \phi_2$ and $\eta:\phi_2 \ra \phi_3$ be isogenies.
Then $\deg\eta\mu=\deg\eta\deg\mu$.
\end{lem}
\begin{proof}
It follows from the exact sequence 
$
0 \ra \Ker \mu \hra \Ker\eta\mu \overset{\mu}{\ra} \Ker\eta \ra 0.
$
\end{proof}

We see that  $\deg\mu=A$ if and only if $\mu$ is an isomorphism.
If $\deg\mu=\frn$, then it is called an {\it $\frn$-isogeny}.
If $\Ker\mu \cong A{/}\frn$, then it is called a {\it cyclic $\frn$-isogeny}.
For example, 
if $\frn=(a)$ for some $a \neq 0$ and  $\rk\phi=r$, then $\deg\phi_a=\frn^r$ since $\Ker \phi_a=\phi[\frn] \cong (A/\frn)^{\oplus r}$.
Hence
$\phi_a$ is an $\frn^r$-isogeny.

A finite $A$-submodule $\Lambda \subset {}_\phi\sepK$ with $\Lambda \cong A/\frn$ is called a {\it cyclic $\frn$-kernel} of $\phi$.
For such $\Lambda$, there is  a cyclic  $\frn$-isogeny 
$\mu:\phi \ra \psi$  with $\Ker\mu=\Lambda$ (cf. \cite[pp.\ 37]{DH87}) and $\mu$ is unique up to $\sepK$-equivalence.
If $\Lambda$ is $G_L$-stable for some field $L/K$, then $\Lambda$ is said to be {\it $L$-rational}.
It follows  that  $\Lambda$ is $L$-rational if and only if  $\mu$ can be taken to be $L$-rational.

\begin{prop}\label{propdual}
Let $\mu:\phi \ra \psi$ be a $K$-rational isogeny  and suppose that a non-zero $a \in A$  annihilates $\Ker\mu$.
Then there exists a unique  isogeny  
$
\eta:\psi \ra \phi
$
such that it is $K$-rational and 
 \[
\begin{array}{ccc}
\eta \mu=\phi_{a} & \mbox{and} &  \Ker \eta=\mu(\phi[a]), \\
 \mu \eta=\psi_{a} & \mbox{and} & \Ker \mu=\eta(\psi[a]).
\end{array}
\]
\end{prop}

\begin{proof}
The existence of $\eta$ is well-known; see \cite[\S 4.7]{Goss} for example.
The uniqueness follows from 
the right division algorithm \cite[Proposition 1.6.2]{Goss}.
\end{proof}

Let $\mu:\phi \ra \psi$ be an $\frn$-isogeny and let $a_\frn \in A$ be  the monic generator of $\frn$.
Then  there exists a unique isogeny 
\[
\hat{\mu}:\psi \ra \phi
\] 
such that  $\hat{\mu}\mu=\phi_{a_\frn}$ and $\mu\hat{\mu}=\psi_{a_\frn}$ by Proposition \ref{propdual}.
We call $\hat{\mu}$    the {\it dual isogeny} of $\mu$.

\begin{rem}\label{remdegree}
If $\phi$ is of rank $r$, then $
\frn^r=\deg\phi_{a_\frn}=\deg\hat{\mu}\deg\mu=\frn\deg\hat{\mu}
$.
 Hence 
$
\deg\hat{\mu}=\frn^{r-1}.
$
In particular, we have  $\deg\mu=\deg\hat{\mu}$ if $\rk\phi=2$.
\end{rem}

By Proposition \ref{propdual}, 
any $K$-rational isogeny $\mu:\phi \ra \psi$ has the inverse $\mu^{-1}  $ in $\Hom{K}{\psi}{\phi}{\otimes_A}Q$ and hence  
  we have  the isomorphisms 
\[
\begin{array}{ccccc}
\End{K}{\phi}{\otimes_A}Q & \cong & \Hom{K}{\phi}{\psi}{\otimes_A}Q & \cong & \End{K}{\psi}{\otimes_A}Q \\
f & \mapsto  & \mu f & \mapsto  &\mu f\mu^{-1} 
\end{array}
\]
of $Q$-vector spaces.
Hence if $\phi$ has complex multiplication, then so does $\psi$.

\begin{definition}\label{defprimitive}
A $K$-rational  isogeny $\mu: \phi \ra \psi$ is said to be  {\it primitive} if there are no $K$-rational isogenies 
$\eta:\phi \ra \psi$ satisfying  $\deg\eta \mid \deg\mu$ and $\deg\eta \neq \deg\mu$.
\end{definition}

Let $\mu:\phi \ra \psi$ be a $K$-rational isogeny between non-CM Drinfeld $A$-modules.
Then Corollary \ref{cor:primitive} as below implies that
 $\mu$ is primitive if and only if it is cyclic.
In this case, it generates $\Hom{K}{\phi}{\psi}$ as an $A$-module and so every   $\eta \in \Hom{K}{\phi}{\psi}$ is given by $\eta=\mu\phi_a$ for some $a \in A$.

\begin{lem}\label{propnoncmdegree}
Let $\mu_1$ and $\mu_2$ be isogenies in $\Hom{K}{\phi}{\psi}$.
Suppose that $\phi$ has no CM.
Then $\deg\mu_1=\deg\mu_2$ if and only if $\mu_1=\xi\mu_2$ for some $\xi \in \bF_q^\times$.
\end{lem}

\begin{proof}
By the absence of CM, $\mu_1=\mu\phi_{a_1}$ and $\mu_2=\mu\phi_{a_2}$ hold for some $a_i \in A \ (i=1,2)$, where $\mu$ is a primitive isogeny. 
Then $\deg\phi_{a_1}=\deg\phi_{a_2}$ if and only if $a_2=\xi a_1$ for some $\xi \in \bF_q^\times$.
\end{proof}


\section{Virtually $K$-rational Drinfeld $A$-modules}

Let  $s \in G_K$.
For any  $\mu=\sum_{i=0}^n c_i\tau^i \in \ad{K^{\rm sep}}$, set 
$
{}^s\mu:=\sum_{i=0}^n{}^sc_i\tau^i.
$
Then $\ad{\sepK} \ra \ad{\sepK};\ \mu \mapsto {}^s\mu$ is a 
 ring automorphism of $\ad{\sepK}$.
For a Drinfeld $A$-module $\phi$ over $\sepK$, we define 
a new Drinfeld $A$-module ${}^s\phi$ by 
\[
\begin{array}{cccc}
{}^s\phi: & A & \ra & \ad{K^{\rm sep}}. \\
 & a & \mapsto &  {}^s\phi_a
\end{array}
\]
We call ${}^s\phi$ the {\it conjugate} of $\phi$ by $s$.
Clearly $\rk{}^s\phi=\rk\phi$.
For any  $\sepK$-rational isogeny  $\mu: \phi \ra \psi$, we have 
${}^s\mu {}^s\phi_T^{}={}^s(\mu\phi_T^{})={}^s(\psi_T^{}\mu)={}^s\psi_T^{}{}^s\mu$.
Hence ${}^s\mu$ is an isogeny  
$
{}^s\mu:{}^s\phi \ra {}^s\psi
$.
Then $\mu \mapsto {}^s\mu$ yields  an $A$-module isomorphism $\Hom{\sepK}{\phi}{\psi} \overset{\sim}{\ra} \Hom{\sepK}{{}^s\phi}{{}^s\psi}$.
We have  $\deg\mu=\deg{}^s\mu$ because $s$ induces an isomorphism $\Ker\mu \cong \Ker{}^s\mu$ of $A$-modules.

\begin{prop}
Let $\phi$ and $\psi$ be  Drinfeld $A$-modules over $K$ isogenous to each other.
Suppose that  $\phi$ has no CM over $\sepK$ $($i.e., $\End{\sepK}{\phi}\cong A)$.
Then there are an element $\lambda \in K^\times$  and a positive integer $n$ with $n\mid q-1$  such that any isogeny $\mu:{\phi} \ra {\psi}$ are $K(\sqrt[n]{\lambda})$-rational, where $\sqrt[n]{\lambda}$ is an $n$-th root of $\lambda$.
\end{prop}

\begin{proof}
It is enough to prove for   a primitive isogeny $\mu: \phi \ra \psi$ in $\Hom{\sepK}{\phi}{\psi}$. 
Set  $\mu=c_0^{} + c_1^{}\tau+ \cdots + c_N^{}\tau^N \in \ad{\sepK}$.
Recall that  $c_0 \neq 0$.
Let $s \in G_K$.
Since  $\deg\mu=\deg{}^s\mu$,  there is a unique $\xi_s \in \bF_q^\times$ such that   ${}^s\mu=\xi_s\mu$ by
Lemma \ref{propnoncmdegree}.
Thus we have ${}^sc_i=\xi_sc_i$ for any $0 \leq i \leq N$.
Consider the group homomorphism
\[
\begin{array}{cccc}
\chi: & G_K & \ra &\bF_q^\times .\\
 & s & \mapsto &  \xi_s
\end{array}
\]
Denote by $n$ the order of $\chi$, so that  $n \mid q-1$.
Then  $\lambda:=c_0^n$ satisfies  ${}^s\lambda=\xi_s^n\lambda=\lambda$ for any $s \in G_K$ and hence $\lambda \in K^\times$.
Set  $c_i':=\frac{c_i}{c_0}$ for each $1\leq i \leq N$.
Then  $c_i' \in K$ since ${}^sc_i'=\frac{\xi_sc_i}{\xi_sc_0}=c_i'$ for any $s \in G_K$.
This implies that   $\mu=c_0\mu'$ for  $\mu'=1 + c_1'\tau+ \cdots + c_N'\tau^N \in \ad{K}$.
Hence $\mu$ is rational over $K(c_0)=K(\sqrt[n]{\lambda})$.   
\end{proof}

We say that a Drinfeld $A$-module $\phi$ is {\it $K$-rational} if it is $\bar{K}$-isomorphic to some $\psi$ defined over $K$, and then $\psi$ is called a {\it $K$-model} of $\phi$.
The following rationality criterion is   an analogue of Weil's classical result \cite{Wei56} for  algebraic varieties.

\begin{thm}\label{thmisom}
For a Drinfeld $A$-module  $\phi$ defined over $\sepK$, the following are equivalent.
\begin{itemize}
\item[{\rm (1)}] $\phi$ is  $K$-rational,
\item[{\rm (2)}]  There exists an isomorphism $\nu_s:{}^s\phi \ra \phi$ for each $s \in G_K$ such that ${}^s\nu_t\cdot\nu_s=\nu_{st}$ for any $s, t \in G_K$.
\end{itemize}
\end{thm}

\begin{proof}
We first prove (1)$\Rightarrow$(2).
Let $\psi$ be a $K$-model of $\phi$ and 
take an isomorphism $\nu:\phi \ra \psi$, so that $\psi_T^{}=\nu^{}\phi_T^{}\nu^{-1}$. 
Since ${}^s\psi=\psi$  for any $s \in G_K$, ${}^s\nu$  gives raise to  an isomorphism ${}^s\nu:{}^s\phi \ra  \psi$.
For each $s\in G_K$, define an isomorphism $\nu_s:{}^s\phi \ra  \phi$ by $\nu_s:=\nu^{-1}\cdot{}^s\nu$.
Then 
$
{}^s\nu_t\cdot\nu_s=\nu_{st} 
$ for any $s,t \in G_K$.

To prove (2)$\Rightarrow$(1), take a family
 $\{\nu_s:{}^s\phi \ra \phi\}_{s\in G_K}$ of isomorphisms as in (2).
Since $\phi$ is actually defined over a finite extension $L/K$, we may assume that 
$\nu_{s}=\nu_{t}$ if $s|_L=t|_L$.
 Then the  map
$\alpha:G_K \ra  {\sepK}^{,\times};\ s \mapsto \nu_s$ is continuous with respect to 
the Krull topology on $G_K$ and the discrete topology on  ${\sepK}^{,\times}$.
Since $\alpha$ satisfies the one-cocycle condition, 
Hilbert's theorem 90 implies that 
there is an element $\nu  \in {\sepK}^{,\times}$ such that $\nu_s=\nu^{-1}\cdot{}^s\nu$ for any $s\in G_K$. 
Let $\psi$ be the   Drinfeld $A$-module determined  by $\psi_T^{}=\nu \phi_T^{} \nu^{-1}$.
Then for any $s \in G_K$, we have
\[
{}^s\psi_T^{}={}^s\nu{}^s\phi_T^{} {}^s\nu^{-1}=    
\nu  \nu_s \cdot \nu_s^{-1}\phi_T^{}\nu_s \cdot \nu_s^{-1}  \nu^{-1}      
=\psi_T^{}.
\]
 Hence  $\psi$ is   a  $K$-model   of $\phi$.
\end{proof}

As  a function field analogue of  $\bQ$-curves, we consider the following weak rationality of Drinfeld $A$-modules.

\begin{definition}\label{defKvirtual}
A  Drinfeld $A$-module $\phi$ is said to be  {\it virtually $K$-rational} if 
it  is defined over $\sepK$ and isogenous to  ${}^s\phi$ for any $s \in G_K$.
\end{definition}

Let $\phi$ be defined over $\sepK$.
By Theorem \ref{thmisom}, if  $\phi$ is $K$-rational, then it  is virtually $K$-rational.
If $\rk\phi=1$,   then $\phi$ is virtually $K$-rational  since all rank-one Drinfeld $A$-modules are isomorphic  to  each other over $\bar{K}$.
For this reason,  our interest focuses on the non-$K$-rational and  $\rk \phi \geq 2$ case.
The following gives a non-trivial example of virtually $Q$-rational Drinfeld $A$-modules.

\begin{ex}\label{exKvirtual}
Suppose that $p \neq 2$ and fix a square root $\sqrt{T+1} \in Q^{\rm sep}$ of $T+1 \in A$. 
Define two elements of $\ad{Q^{\rm sep}}$ by 
$\mu:=\sqrt{T+1}+1-\tau$ and $\eta:=\sqrt{T+1}-1+\tau$. 
Then  
\begin{eqnarray}\nonumber
\mu\eta &=& (\sqrt{T+1}+1-\tau)(\sqrt{T+1}-1+\tau) \\ \nonumber 
 &=& T + (2+\sqrt{T+1}-\sqrt{T+1}^q_{})\tau-\tau^2. \nonumber
\end{eqnarray}
Let  $\varphi$ be  the  Drinfeld $A$-module determined by  $\varphi_T^{}=\mu\eta$, so that 
$\rk \varphi=2$.
Then  it is not $Q$-rational but virtually $Q$-rational as follows.
If $s\in G_Q$ fixes $\sqrt{T+1}$, then  ${}^s\varphi=\varphi$. 
If  ${}^s\sqrt{T+1}=-\sqrt{T+1}$,
then
\begin{eqnarray}\nonumber
{}^s\varphi_T^{}&=&{}^s\mu{}^s\eta \\ \nonumber
&=& (-\sqrt{T+1}+1-\tau)(-\sqrt{T+1}-1+\tau) \\ \nonumber 
 &=& (\sqrt{T+1}-1+\tau)(\sqrt{T+1}+1-\tau) \\ \nonumber
&=&\eta\mu. \nonumber
\end{eqnarray}
Thus ${}^s\varphi$ is isogenous to $\varphi$ because  $\mu{}^s\varphi_T^{}=\mu\eta\mu=\varphi_T^{}\mu$.
Hence $\varphi$ is virtually $Q$-rational. 
Now the $j$-invariant  
\[
j_\varphi=- (2+\sqrt{T+1}-\sqrt{T+1}^q_{})^{q+1}
\]
of $\varphi$ is  not contained in $Q$.
 Hence $\varphi$ is not $Q$-rational   by the next remark.
\end{ex}

\begin{rem}
For a rank-two Drinfeld $A$-module determined by $\phi^{}_T=T+g\tau+\Delta\tau^2 \in \ad{\bar{K}}$, its {\it $j$-invariant} is defined by 
\[
j_\phi=\frac{g^{q+1}}{\Delta}.
\]
It follows that $\phi$ is $\bar{K}$-isomorphic to some $\psi$ if and only if $j_\phi=j_\psi$.
In particular, $\phi$ is $Q(j_\phi)$-rational   because the Drinfeld $A$-module $\phi'$ determined by $\phi'_T=T+j_\phi\tau+j_\phi^q\tau^2$ has the $j$-invariant  $j_\phi$.
Hence $\phi$ is $K$-rational  if and only if $j_\phi \in K$.
\end{rem}

In the full CM case, the explicit class field theory (cf.\ \cite{Hay79} and  \cite{Hay92}) implies  the following.

\begin{prop}\label{propCMKvirtual}
Let $\phi$ be a Drinfeld $A$-module over $\bar{K}$  with
full CM by a finite extension $F/Q$.
Then $\phi$ is isogenous to a  virtually $F$-rational Drinfeld $A$-module.
\end{prop}

\begin{proof}
Let $\cO_F$ be the ring of integers of $F$.
By replacing $\phi$ with a suitable  isogenous one if necessarily, we may assume that 
$\End{\bar{K}}{\phi} \cong \cO_F$ by  \cite[Proposition 4.7.19]{Goss}.
Then we obtain a Drinfeld $\cO_F$-module $\Phi:\cO_F \ra \ad{\bar{K}}$ of rank one with
$\Phi|_A=\phi$. 
Recall that there is a unique place $\infty_F$ of $F$ lying above $\infty=(1/T)$.
Let $H_F$ be the Hilbert class field of $F$.
Namely it is the maximal unramified abelian extension of $F$ in which $\infty_F$ splits completely.
Then it is known that   $\Phi$  is $H_F$-rational. 
Let $\Psi$ be an  $H_F$-model of $\Phi$. 
Then the restriction  $\psi:=\Psi|_{A}$  is   an $H_F$-model of  $\phi$. 
Let $s \in G_F$.
Then we have  
$s|_{H_F}=\mathrm{Frob_{\frP}}$ for some Frobenius automorphism $\mathrm{Frob}_{\frP} \in \Gal{H_F}{F}$
at   $\frP \subset \cO_F$, so that  ${}^s\Psi = {}^{\mathrm{Frob_\frP}}\Psi$.
By  \cite[Theorem 10.8]{Hay92},  the conjugate  ${}^{\mathrm{Frob_\frP}}\Psi$ is isomorphic to
the Drinfeld $\cO_F$-module 
 $\frP{*}\Psi$ given by the action of ideas on $\Psi$, which is isogenous to $\Psi$; see 
\cite[pp.7]{Hay92}.
Hence we have an isogeny ${}^s\Psi \ra \Psi$ and it yields an isogeny ${}^s\psi \ra \psi$.
\end{proof}


\section{The modular curve $Y_*(\frn)$}

Let  $\frn \subset A$ be  a non-zero ideal. From now on, we assume that  any Drinfeld $A$-module is of rank two.
This section is devoted to a study of rational points of  the Drinfeld modular curve $Y_0(\frn)$ of $\Gamma_0(\frn)$-level and its quotient $Y_*(\frn)$.
Similar to  elliptic modular curves,  we see that $Y_0(\frn)$ is affine smooth over $Q$ and  has a unique smooth compactification  $X_0(\frn)$.
See \cite{Gek86}, \cite{GR96}, \cite{Gek01}, and  \cite{Sch} for more details.

Recall that 
$Y_0(\frn)$ is a coarse moduli variety of (rank-two) Drinfeld $A$-modules with additional structures.
More preciously,  
every  $K$-rational point  $x \in Y_0(\frn)(K)$ corresponds to  a $\bar{K}$-equivalence class
of a $K$-rational cyclic $\frn$-isogeny. 
Equivalently, every $x \in Y_0(\frn)(K)$ corresponds to 
a $\bar{K}$-isomorphism class $[\phi,\Lambda]$ of a pair $(\phi, \Lambda)$ consisting of   a Drinfeld $A$-module  $\phi$ over $K$   and  a   $K$-rational cyclic $\frn$-kernel $\Lambda$  of $\phi$.
Here two such pairs $(\phi,\Lambda)$ and $(\phi',\Lambda')$ are said to be {\it  $\bar{K}$-isomorphic} if there exists an $\bar{K}$-isomorphism $\nu:\phi \ra \phi'$  such that $\nu \Lambda=\Lambda'$.
If $x \in Y_0(\frn)(K)$ is represented by a $K$-rational cyclic $\frn$-isogeny $\mu:\phi \ra \psi$ with $\Ker \mu=\Lambda$, then  we  use the notation $x=[\phi, \Lambda]$ or $x=[\mu:\phi \ra \psi]$.
We say that $x$ is a {\it CM point} if $\phi$ has CM.

\begin{rem}
Let $\ell$ be a prime number.
Then Mazur's famous result \cite[Theorem 7.1]{Maz78} asserts that the elliptic modular curve $Y_0(\ell)$ has 
no $\bQ$-rational points if $\ell > 163$.
As a partial  analogue of this, the following is known:
\begin{thm}[{\cite[Theorem 1.2]{Pal10}}]
Let $\frp \subset A$ be  a non-zero  prime ideal and  $a_\frp \in A$ a generator of $\frp$. 
Assume that $Q=\bF_2(T)$.
Then $Y_0(\frp)$ has no $Q$-rational points if $\deg(a_\frp) \geq 3$.
\end{thm}
\end{rem}

 Next we review some properties of Atkin-Lehner involutions; see \cite{Sch} for detail.
Let $\frm \subset A$ be a non-zero ideal with   $\frm \mid \frn$ and  $(\frm, \frac{\frn}{\frm})=1$.
Then  there is  an involution 
\[
w_\frm:Y_0(\frn) \overset{\sim}{\lra} Y_0(\frn)
\]
defined over $Q$, so-called  the   ({\it partial}) {\it  Atkin-Lehner involution} with respect to $\frm$. 
If $\frm=\frn$, then  $w_\frn$ is sometimes called the {\it full Atkin-Lehner involution}.
Denote by   $\cW(\frn)$ the group consisting  of all Atkin-Lehner involutions.
Since  
\[
w_{\frm_1}w_{\frm_2}=w_{\frm_2}w_{\frm_1}=w_{\frm_3} \ \mbox{for}\  \frm_3=\frac{\frm_1\frm_2}{(\frm_1,\frm_2)^2},
\]
 we have  $\cW(\frn) \cong (\bZ{/}2\bZ)^n$, where $n$ is the number of distinct prime factors of $\frn$.

Let  $w_\frm \in \cW(\frn)$ and let $x=[\mu:\phi \ra \psi] \in  Y_0(\frn)(K)$ be represented by  a $K$-rational $\mu$.
Then
 the moduli interpretation of $w_\frm x$  is  as follows.
If  $\frn=\frm\frn'$ with  $(\frm, \frn')=1$,
then $\Ker \mu =\Lambda_{\frm} \oplus \Lambda_{\frn'}$ with  
 $\Lambda_\frm \cong A{/}\frm$ and $\Lambda_{\frn'} \cong A{/}\frn'$. 
Hence $\mu$ decomposes as $\mu=\mu_{\frn'}\mu_\frm$, where $\mu_\frm:\phi \ra \phi_\frm$ is a $K$-rational cyclic $\frm$-isogeny with $\Ker \mu_\frm=\Lambda_\frm$ and $\mu_{\frn'}:\phi_\frm \ra \psi$ is a $K$-rational cyclic $\frn'$-isogeny with $\Ker\mu_{\frn'}=\mu_\frm(\Ker\mu)=\mu_\frm(\Lambda_{\frn'})$.
Then 
\[ 
w_\frm x= 
\left[\phi_\frm, \ \mu_\frm(\phi[\frm] \oplus \Lambda_{\frn'})\right]=[\eta:\phi_\frm \ra \psi_\frm],
\]
where $\eta:\phi_\frm \ra \psi_\frm$ is a $K$-rational cyclic $\frn$-isogeny with $\Ker\eta=\mu_\frm(\phi[\frm] \oplus \Lambda_{\frn'})$.
Hence we have  $w_\frm x \in Y_0(\frn)(K)$.
Let us  decompose  $\eta$ as $\eta=\eta_{\frn'}\eta_\frm$ with $\eta_\frm:\phi_\frm \ra \phi'$ and $\eta_{\frn'}:\phi' \ra \psi_\frm$ similarly as $\mu=\mu_{\frn'}\mu_\frm$. 
Then $\Ker\eta_\frm\mu_\frm=\phi[\frm]$ by construction.
This  implies  that  $\lambda\eta_\frm\mu_\frm=\phi_{a_\frm}$ for some $\lambda \in K^\times$, where $a_\frm$ is the monic generator of $\frm$.
\begin{lem}
The  $\lambda$ gives raise to  an isomorphism $\lambda:\phi' \ra \phi$ satisfying $\lambda^{}\eta_\frm=\hat{\mu}_\frm$.
\end{lem}

\begin{proof}
Since $\phi'_T\eta_\frm\mu_\frm=\eta_\frm\mu_\frm\phi_T$ holds, the equation $\lambda\eta_\frm\mu_\frm=\phi_{a_\frm}$ implies
\[
\lambda\phi'_T\lambda^{-1}\phi_{a_\frm}=
\lambda\phi'_T \eta_\frm\mu_\frm=
\lambda\eta_\frm\mu_\frm\phi_T^{} =
\phi_{a_\frm}\phi_T^{}=
\phi_T^{}\phi_{a_\frm}.
\]
Hence $\lambda\phi'_T\lambda^{-1}=\phi_T^{}$ by the right division algorithm and so  $\lambda$ is an isomorphism.
The equality  $\lambda^{}\eta_\frm = \hat{\mu}_\frm$ follows from the uniqueness of the dual isogeny; see Proposition \ref{propdual}.
\end{proof}

Thus we obtain the following commutative diagram

\begin{equation}\label{diagram}
\xymatrix{
\phi \ar[rr]^-\mu \ar[rd]_-{\mu_\frm}& & \psi \\
 & \phi_\frm \ar[rr]^-\eta \ar[rd]_-{\hat{\mu}_\frm} \ar[ru]^{\mu_{\frn'}}& & \psi_\frm \\
 & & \phi \ar[ru]_-{\eta_{\frn'}\lambda^{-1}}& \hspace{20pt}.
}
\end{equation}
If in particular $\frm=\frn$, then   by construction $\phi_\frm=\psi$.
Hence $
w_\frn x
$ 
 is represented by the dual  isogeny $\hat{\mu}:\psi \ra \phi$.

\begin{rem}\label{remfixedpt}
For $x \in Y_0(\frn)(\bar{K})$, consider the {\it decomposition group} 
\[
D_x:=\{w \in \cW(\frn); wx=x\}
\]
 of $x$.
It is known that  the number of points in $Y_0(\frn)(\bar{K})$ fixed by some $w \in \cW(\frn)$ is finite and so the group $D_x$ is trivial for almost all $x$.
In particular, $D_x$ is trivial if $x$ is a non-CM point since any point fixed by a non-trivial involution in $\cW(\frn)$ is a CM point. 
For more details, let $w_\frm \in \cW(\frn)$ be a non-trivial involution and denote by $a_\frm \in \frm$ the monic generator.
Let 
$x=[\phi, \Lambda] \in Y_0(\frn)(\bar{K})$ and suppose that $w_\frm x=x$.
If $q$ is odd, then $\phi$ has CM by $Q(\sqrt{\xi a_\frm})$ for some $\xi \in \bF_q^\times$; see \cite{Gek86} or \cite{Sch}.
This implies that $D_x=\{\id, w_\frm \}$.
On the other hand, if $q$ is even, then 
 $\phi$ has CM by the inseparable extension  $Q(\sqrt{T})$; see \cite[pp. 338]{Sch} for example. 
Hence  $D_x$ may become larger.
\end{rem}

Denote by $Y_0(1)$ the Drinfeld modular curve for the ideal $(1)=A$.
Then any $x \in Y_0(1)(K)$ corresponds to a $\bar{K}$-isomorphism class $[\phi]$ of a Drinfeld $A$-module $\phi$ over $K$.
 Let  $\theta:Y_0(\frn) \ra Y_0(1)$ be the morphism given by forgetting the  level structure.
Then we  have  $\theta(x)=[\phi]$ for  $x=[\mu:\phi \ra \psi] \in Y_0(\frn)(K)$.
 Define 
$
\cN_{0}(\frn)(\sepK) \subset Y_0(\frn)(\sepK)
$ 
to be the subset consisting of  all non-CM $\sepK$-rational  points of $Y_0(\frn)$. 
Consider the map 
\[
\Theta: \cN_{0}(\frn)(\sepK)  \ra   Y_0(1)(\sepK) \times Y_0(1)(\sepK) 
\]
defined by  $\Theta( x  ) = \left(\theta (x) ,  \theta(w_\frn x)\right)$. 
The following  lemma is needed in Section 5.

\begin{lem}\label{lemmodular}
The map $ 
\Theta
$
 is injective. 
\end{lem}

\begin{proof}
Take two points $x, y \in \cN_0(\frn)(\sepK)$ with $x=[\mu:\phi \ra \psi]$ and $y=[\eta:\phi' \ra \psi']$, where 
both $\mu$ and $\eta$ are $\sepK$-rational.
Assume that  $\Theta(x)=\Theta(y)$.
Since $w_\frn x$ and $w_\frn y$ are represented by $\hat{\mu}$ and $\hat{\eta}$ respectively, we have
$[\phi]=[\phi']$ and $[\psi]=[\psi']$.
 Thus we can take $\sepK$-isomorphisms
 $\nu:\phi \ra \phi'$ and $\lambda:\psi \ra \psi'$.
Then $\eta':=\lambda\mu\nu^{-1}$ yields a $\sepK$-rational cyclic $\frn$-isogeny $\eta':\phi' \ra \psi'$.
Since $\deg\eta=\deg\eta'$, 
Lemma \ref{propnoncmdegree} implies that  $\eta=\xi\eta'$ for some $\xi \in \bF_q^\times$. 
Hence  $\eta=(\xi\lambda)\mu\nu^{-1}$ and so $\mu$ and $\eta$ are $\bar{K}$-equivalent.    
Thus $x=y$.
\end{proof}

Let
$Y_*(\frn):=Y_0(\frn)/\cW(\frn)$ be the quotient of $Y_0(\frn)$ by all Atkin-Lehner involutions.
Then it is an affine curve over $Q$.
 Denote 
 by $\gamma:Y_0(\frn) \ra Y_*(\frn)$ the quotient  morphism, which is defined over $Q$.
Since $Y_0(\frn)$ is quasi-projective, $\gamma$ is  finite and  $\cW(\frn)$ acts transitively on the fibers of $\gamma$; see \cite[pp.113]{Liu02}. 
Therefore for any $x_* \in Y_*(\frn)(K)$, the pre-image $\cP(x_*):=\gamma^{-1}(x_*) \subset Y_0(\frn)(\bar{K})$  is the $\cW(\frn)$-orbit  $\cW(\frn)x$ of  some  $x \in Y_0(\frn)(\bar{K})$.
We consider the following condition for $x_*$:

\begin{equation}
 \text{$\cP(x_*)$ is contained in $Y_0(\frn)(\sepK)$.} \tag{$*$}
\end{equation}
This is equivalent to  that $\cP(x_*)$ contains at least one $\sepK$-rational point of $Y_0(\frn)$. 
If ($*$) holds, then $\cP(x_*)$  is $G_K$-stable since $\gamma({}^sx)={}^s\gamma(x)={}^sx_*=x_*$ for any $x\in\cP(x_*)$ and $s\in G_K$.

\begin{lem}\label{lemgstable}
Let $x_* \in Y_*(\frn)(K)$ and assume either $q$ is odd or $\cP(x_*)$ has no  CM points.
Then $x_*$ satisfies $(*)$.
\end{lem}
\begin{proof}
Take $x \in \cP(x_*)$, so that $\cP(x_*)=\cW(\frn)x$.
Denote by $\kappa(x)$ and $\kappa(x_*)$ the residue fields at $x$ and $x_*$, respectively.
To check the condition $(*)$,  it is enough to show that the extension $\kappa(x)/\kappa(x_*)$ is separable.
Let $\tilde{\gamma}:Y_0(\frn)\ra Y_0(\frn)/D_x$ be the quotient by the decomposition group $D_x$ of $x$.
Then $\gamma:Y_0(\frn)\ra Y_*(\frn)$ factors as
\[
\xymatrix@C=36pt{
Y_0(\frn) \ar[rr]^-{\gamma} \ar[rd]_-{\tilde{\gamma}} &  &Y_*(\frn)  \\
 &  Y_0(\frn){/}D_x\ar[ru] &
}
\]
such that $Y_0(\frn){/}D_x \ra Y_*(\frn)$ is \'etale at $\tilde{x}:=\tilde{\gamma}(x)$; see \cite[pp.147]{Liu02}.
The assumption and Remark \ref{remfixedpt} imply that $[\kappa(x):\kappa(\tilde{x})]\leq 2$ and that $\kappa(x)=\kappa(\tilde{x})$ if $x$ is a non-CM point. 
Hence $\kappa(x)/\kappa(\tilde{x})$ is separable and so is $\kappa(x)/\kappa(x_*)$.
\end{proof}

\begin{prop}\label{rat}
If $x_* \in Y_*(\frn)(K)$ satisfies $(*)$, then any  $x \in \cP(x_*)$ is represented by a pair $(\phi, \Lambda)$ such that 
$\phi$ is virtually $K$-rational and  defined over a polyquadratic extension of $K$.
\end{prop}

\begin{proof}
Take a point $x=[\phi, \Lambda] \in \cP(x_*)$, where $\phi$ is defined over $\sepK$ and $\Lambda$ is $\sepK$-rational.
Let $s \in G_K$.
 Then  ${}^sx$ is represented by $({}^s\phi,{}^s\Lambda)$.
We see that 
 there is an involution $w_{\frm_s} \in \cW(\frn)$ satisfying ${}^sx=w_{\frm_s}x$ because    $\cP(x_*)=\cW(\frn)x$  is $G_K$-stable.
Hence  ${}^s\phi$   admits a cyclic $\frm_s$-isogeny to  $\phi$ and so $\phi$ is virtually $K$-rational.
Now the above correspondence $s \mapsto w_{\frm_s}$ induces a well-defined group homomorphism 
\[
f:G_K \ra \cW(\frn){/}D_x.
\]
Let $L \subset \sepK$ be the fixed subfield of $\Ker f$.
Then it is polyquadratic over $K$ since $\Gal{L}{K}$ injects into $\cW(\frn){/}D_x \cong (\bZ/2\bZ)^m$ for some $m \geq 0$. 
Then $x$ an $L$-rational point and hence $\phi$ has an $L$-model. 
\end{proof}

\begin{rem}\label{rem:central}
Let $\phi$ be virtually $K$-rational  and $\frn \subset A$ a non-zero square-free ideal.
Following the $\bQ$-curve case, we say that  $\phi$ is {\it central} of degree $\frn$ if ${}^s\phi$ admits an isogeny ${}^s\phi \ra \psi$ of degree dividing $\frn$ for any $s\in G_K$. 
Theorem \ref{thmmain} and  Proposition \ref{rat} imply that  if $\phi$ has no CM, then it is isogenous to a central one defined over a polyquadratic extension of  $K$.
Thus the existence of non-CM $K$-rational points of $Y_*(\frn)$ is equivalent to that of 
non-CM central virtually $K$-rational Drinfeld $A$-modules of rank two of degree $\frn$.
\end{rem}


\section{Isogeny trees}

In this final section, we prove  Theorem \ref{thmmain}.
For the rest of this paper,  the terminology ``Drinfeld $A$-module'' always refers to rank-two  non-CM one defined  over $\sepK$.
We use the symbol $\frp$ for a non-zero prime ideal of $A$.

Let $\frp \subset A$ be a non-zero prime ideal and $n>0$ a positive  integer.
We first recall  some facts on $\frp^n$-isogenies.
Let $\phi$ and $\psi$ be Drinfeld $A$-modules.

\begin{lem}\label{lemdecomp}
Every cyclic $\frp^n$-isogeny $\mu:\phi \ra \psi$ factors as $\mu=\mu_n\mu_{n-1}\cdots \mu_1$ with some $\frp$-isogenies $\mu_i$.
The $\mu_i$ are unique up to $\sepK$-equivalence.
\end{lem}
\begin{proof}
We prove this  by induction on $n$.
The case where $n=1$ is trivial.
Assume that $n \geq 2$ and  $\deg\mu=\frp^n$.
Since  $\Ker\mu$ contains a unique (cyclic) $\frp$-kernel $\Lambda$, 
$\mu$ factors as $\mu=\mu_{n-1}\mu_{1}$ such that  $\Ker\mu_1=\Lambda$ and $\mu_{n-1}$ is of degree ${\frp^{n-1}}$.
Since $\mu_1$ is uniquely determined up to $\sepK$-equivalence, we obtain the conclusion.
\end{proof}

\begin{prop}\label{proppc}
Let $\mu:\phi \ra \psi$ be a $\frp^n$-isogeny. 
Then it is cyclic if and only if it is primitive.
\end{prop}

\begin{proof}
Take a primitive isogeny  $\eta:\phi \ra \psi$.
Then  $\mu=\eta\phi_a$ for some non-zero $a \in A$ since $\phi$ has no CM.
Suppose that $\mu$ is cyclic.
Considering the degree of $\mu=\eta\phi_a$, we have either 
$a \in \frp$ or $a \in \bF_q^\times$.
If $a \in \frp$, then  $\Ker\mu=\eta^{-1}(\psi[a])$ is not cyclic.
Hence   $a \in \bF_q^\times$ and so $\mu$ is  primitive.
Conversely, if  $\mu$ is not cyclic, then  $\Ker\mu=\Lambda_1 \oplus \Lambda_2$ for some non-trivial cyclic $\frp$-power kernels $\Lambda_i$ of $\phi$.
Then  $\Ker\mu$ in particular  contains $\phi[\frp]$.
This means that  $\mu$ is not primitive. 
\end{proof}

\begin{cor}\label{cor:primitive}
For any  isogeny $\mu:\phi \ra \psi$, it  is cyclic if and only if it is primitive. 
\end{cor}

\begin{proof}
Set $\deg\mu=\prod_{i=1}^n\frp_i^{\delta_i}$, where $\frp_i$ are distinct prime ideals and $\delta_i>0$. 
Then for each $i$, $\Ker\mu$ contains a unique  $\frp_i^{\delta_i}$-kernel and so $\mu$ decomposes as $\mu=\eta_i\mu_i$ such that 
$\deg\mu_i=\frp_i^{\delta_i}$ and $\deg\eta_i=\prod_{i \neq j}\frp_j^{\delta_j}$.
Applying Proposition \ref{proppc} to each $\mu_i$, we have the conclusion.
\end{proof}

Suppose that  $\phi$ and $\psi$ are isogenous to each other and 
take a primitive isogeny 
$\mu:\phi \ra \psi$.
For any non-zero prime ideal $\frp \subset A$, define
\begin{equation}\label{eqdelta}
\delta_\frp(\phi,\psi) :=\max \{ n \in \bZ_{\geq 0};\ \deg\mu \  \mbox{is divisible by}\ \frp^n\}.
\end{equation}
It is independent of the choice of $\mu$ since we now consider the non-CM case.

\begin{lem}\label{lemdelta}
Let $\phi$, $\psi$ and $\frp$ be as above. Then
 $\delta_\frp(\phi,\psi) = \delta_\frp(\psi,\phi)$ and  $\delta_\frp({}^s\phi,{}^s\psi) = \delta_\frp(\phi,\psi)$  hold for any $s \in G_K$.
\end{lem}

\begin{proof}
It  follows from that  $\deg\mu=\deg\hat{\mu}$ and $\deg{}^s\mu=\deg\mu$ for any $s \in G_K$. 
\end{proof}

\begin{rem}\label{remppart}
Suppose that 
$\delta_\frp(\phi, \psi)=n$.
 Then  there exists a Drinfeld $A$-module $\pi_\frp(\psi)$ such that any primitive isogeny $\mu:\phi \ra \psi$ decomposes as $\mu=\eta\mu_n$, where 
$\mu_n:\phi \ra \pi_\frp(\psi)$ is a $\frp^n$-isogeny and $\eta:\pi_\frp(\psi)\ra \psi$ is of degree prime to $\frp$.
\[
\xymatrix@C=36pt{
\phi \ar[rr]^-{\mu} \ar[rd]_-{\mu_n} &  &\psi  \\
 &  \pi_\frp(\psi)\ar[ru]_-{\eta} &
}
\]
By construction, the $\pi_\frp(\psi)$ is unique up to isomorphisms.
Thus it follows that $\phi$ is isomorphic to $\psi$ if and only if  $\delta_\frp(\phi,\psi)=0$ for all $\frp$.
\end{rem}

From now on, we  identify all isomorphic Drinfeld $A$-modules.
Under this setting, all $\sepK$-equivalent isogenies are identified. 
Notice that  the notions of   degree of isogenies, primitive isogenies and  dual isogenies are well-defined.

We
fix $\phi$ a virtually $K$-rational Drinfeld $A$-module.
Denote by $\I$ the set of Drinfeld $A$-modules isogenous to $\phi$.
Then every $\psi \in \I$ is also virtually $K$-rational and  hence ${}^s\psi$  admits an isogeny to $\phi$ for any $s \in G_K$.
Therefore $G_K$ acts on $\I$ by $(s, \psi) \mapsto {}^s\psi$.

Let $\frp \subset A$ be a non-zero prime ideal.
Consider the function
\[
\delta_\frp: \I \times \I \ra \bZ_{\geq 0}
\]
defined  by (\ref{eqdelta}).
It is symmetric and $G_K$-invariant by Lemma \ref{lemdelta}.
Let 
$
\Ip
$
be  the subset of $\I$ consisting of those admitting $\frp^n$-isogenies to $\phi$  
for some $n \geq 0$.
Then  $G_K$ also acts on $\Ip$. 
For any $\psi \in \I$, let  $\pi_\frp(\psi) \in \Ip$  be  as in Remark \ref{remppart}.
Since  $\pi_\frp(\psi)=\psi$ holds if $\psi \in \Ip$, we obtain the projection
\[
\begin{array}{cccc}
\pi_\frp: & \I & \thra &\Ip 
\end{array}.
\]
For any  $\psi \in \I$, it follows by construction that $\pi_\frp(\psi) = \phi$ for almost all $\frp$.

\begin{lem}\label{lempi} \ 
Let $\psi,\psi_1,\psi_2 \in \I$ and let $\frp$ be as usual. 
\begin{itemize}
\item[{\rm(1)}] $\pi_\frp(\psi)$ is the unique element of $\Ip$  satisfying
$\delta_\frp(\pi_\frp(\psi), \psi) = 0$.
\item[{\rm (2)}] $\delta_\frp(\pi_\frp(\psi_1), \pi_\frp(\psi_2))=\delta_\frp(\psi_1, \psi_2)$.
\item[{\rm(3)}] $\pi_\frp(\psi_1)=\pi_\frp(\psi_2)$ if and only if $\delta_\frp(\psi_1,\psi_2) = 0$. 
In particular, 
$\psi_1=\psi_2$ if and only if $\pi_\frp(\psi_1)=\pi_\frp(\psi_2)$ for all $\frp$.
\end{itemize}
\end{lem}
\begin{proof}
By construction, we have $\delta_\frp(\pi_\frp(\psi), \psi) = 0$. 
Suppose that   $\psi' \in \Ip$ satisfies $\delta_\frp(\psi', \psi) = 0$.
 Then we can take   a primitive isogeny 
$\mu: \pi_\frp(\psi)\ra \psi'$ with  $\deg\mu=\frp^n$ for some $n \geq 0$
since $\pi_\frp(\psi), \psi' \in \Ip$.
Here  $\delta_\frp(\pi_\frp(\phi), \psi) =\delta_\frp(\psi', \psi) = 0$ implies that there is an isogeny $\pi_\frp(\psi)\ra \psi'$ with degree prime to $\frp$.
Hence $n=0$ and so we have (1).
To prove (2), set $n:=\delta_\frp(\pi_\frp(\psi_1), \pi_\frp(\psi_2))$ and take  a primitive $\frp^n$-isogeny 
$\mu:\pi_\frp(\psi_1) \ra \pi_\frp(\psi_2)$.
Let $\tilde{\mu}:\psi_1 \ra \psi_2$ be a primitive isogeny.
Then $\frp^{\delta_\frp(\psi_1,\psi_2)}$ is the maximal $\frp$-power dividing $\deg\tilde{\mu}$.
By the definition of $\pi_\frp$, there exist isogenies $\eta_i:\pi_\frp(\psi_i) \ra \psi_i$ of  degree prime to $\frp$ for $i=1,2$.
Since $\mu$ and $\tilde{\mu}$ are primitive,  we have  $\eta_2 \mu \hat{\eta}_1=a\cdot\tilde{\mu}$  and $\hat{\eta}_2\tilde{\mu}\eta_1=b\cdot\mu $ for some $a, b \in A$.
Comparing the $\frp$-part of the degree of them, we have 
$n \geq \delta_\frp(\psi_1,\psi_2)$ and $n \leq \delta_\frp(\psi_1,\psi_2)$.
Hence $n = \delta_\frp(\psi_1,\psi_2)$. 
The assertion  (3) immediately follows from (2) and Remark \ref{remppart}.
\end{proof}

By Lemma \ref{lempi}, the  projection  $\pi_\frp:\I \thra \Ip$ is compatible  with the $G_K$-action on $\I$.
Indeed, for any  $s \in G_K$ and $\psi \in \I$, we have 
$\delta_\frp({}^s\pi_\frp(\psi), {}^s\psi)=\delta_\frp(\pi_\frp(\psi), \psi)=0$ and hence 
 ${}^s\pi_\frp(\psi)=\pi_\frp({}^s\psi)$ by the uniqueness of $\pi_\frp({}^s\psi)$.
Consider the restricted product 
\[
{\prod_\frp}' \Ip:=\left\{ (\psi_\frp)_\frp \in \prod_\frp \Ip ;\ \psi_\frp=\phi\ \mbox{for almost all}\ \frp \right\}
\]
relative to $\phi$, where $\frp$ runs through all non-zero prime ideals of $A$.
Then 
 the map
\begin{equation}\label{eqglue}
\begin{array}{cccc}
 (\pi_\frp)_\frp : & \I & \lra &  {\prod'_\frp}\Ip \\
 & \psi & \mapsto &  (\pi_\frp(\psi))_\frp
\end{array}
\end{equation}
 is well-defined since  $\pi_\frp(\psi)=\phi$ for almost all $\frp$.

\begin{lem}
The map {\rm(\ref{eqglue})} is bijective.
\end{lem}

\begin{proof}
The injectivity follows from Lemma \ref{lempi} (3).
To prove the surjectivity, take $(\psi_\frp)_\frp \in {\prod'_\frp} \Ip$ and  primitive isogenies $\mu_\frp:\phi \ra \psi_\frp$ for all $\frp$.
Since $\Ker\mu_\frp=0$ for almost all $\frp$, there is an isogeny $\mu:\phi \ra \psi$ with $\Ker\mu=\oplus_\frp\ker\mu_\frp$.
Since  $\psi \in \I$ and $\delta_\frp(\psi_\frp,\psi)=0$ for all $\frp$, we have $\pi_\frp(\psi)=\psi_\frp$.
Hence  {\rm(\ref{eqglue})} maps  $\psi$ to $(\psi_\frp)_\frp$.
\end{proof}

We may regard  $\Ip$ as a graph whose vertices are elements of $\Ip$ and  edges are $\frp$-isogenies between them. 
Moreover, we have:

\begin{prop}\label{proptree}
The graph $\Ip$ is an undirected regular tree of degree $\#( A{/}\frp)+1$.
For any vertices $\pi_\frp(\psi_1), \pi_\frp(\psi_2) \in \Ip$, the length of the path between them is equal to $\delta_\frp(\psi_1,\psi_2)$.
\end{prop}

\begin{proof}
Since the dual of any $\frp$-isogeny is also of degree $\frp$ by Remark \ref{remdegree}, the graph $\Ip$ is undirected.
For any vertex $\pi_\frp(\psi) \in \Ip$, the number of cyclic $\frp$-kernel of $\pi_\frp(\psi)$ is $\# (A{/}\frp) + 1$.
By the absence of CM, such submodules determine distinct $\frp$-isogenies and so $\pi_\frp(\psi)$
 is of degree  $\# (A{/}\frp) +1$.
Take distinct two  vertices  $\pi_\frp(\psi_1), \pi_\frp(\psi_2) \in \Ip$ and 
  a  primitive isogeny $\mu:\pi_\frp(\psi_1) \ra \pi_\frp(\psi_2)$.
Then $\deg\mu=\frp^n$ for some $n > 0$.
Since  $\mu$ is cyclic by Proposition \ref{proppc}, it  uniquely decomposes as
$\mu=\mu_n\mu_{n-1}\cdots\mu_1$ with $\frp$-isogenies $\mu_i$ by Lemma \ref{lemdecomp}.
 Then $\mu$ determines the unique path joining $\pi_\frp(\psi_1)$ and $\pi_\frp(\psi_2)$ and hence  $\Ip$ is a tree. 
By Lemma \ref{lempi} (2), the length of the path is  $\delta_\frp(\psi_1,\psi_2)$.
 \end{proof}

By Lemma \ref{lemdelta},  the $G_K$-action on the set $\Ip$ preserves the length of paths joining any vertices. 
Hence  $G_K$ acts on the tree $\Ip$.
Denote by 
$
\langle \phi \rangle:=\{ {}^s\phi;\ s \in G_K\}
$ the finite subset of $\I$ consisting of all $G_K$-conjugates of $\phi$.
Set
\[
\Lp:=\{ \pi_\frp(^s\phi) ; s \in G_K\} \subset \Ip
\]
for each $\frp$.
Notice that   $\Lp=\{\phi\}$ for almost all $\frp$.

Define $\cT(\Lp)$ to be the minimal finite subtree of $\Ip$ whose vertex set contains $\Lp$.
Such subtree is uniquely determined.   
Since  any terminal vertex of $\cT(\Lp)$ belongs to $\Lp$ and  $G_K$ acts on $\Lp$ as permutations, the  subtree $\cT(\Lp)$ inherits a $G_K$-action from $\Ip$.   
Then  there is a unique vertex or edge of $\cT(\Lp)$ fixed by $G_K$, which is called the {\it center} of $\cT(\Lp)$.
Indeed, the central vertex or edge of a longest path joining two points in $\Lp$ is fixed by $G_K$.
Such a vertex or an edge does not depend on the choice of longest paths.
Hence the center of $\cT(\Lp)$ is well-defined.

For any $\phi' \in \I$, consider the finite set $\pi_\frp\langle\phi'\rangle$ in the same way.
Then for  two subtrees $\T$ and $\cT(\pi_\frp\langle\phi'\rangle)$ of $\Ip$,  
 we have the following.

\begin{lem}\label{lemcenter}
The center of $\T$ is an edge if and only  if the center of $\cT(\pi_\frp\langle\phi'\rangle)$ is an edge.
In this case, the centers of $\T$ and $\cT(\pi_\frp\langle\phi'\rangle)$ coincide.  
\end{lem}

\begin{proof}
Suppose that the center of $\T$ is an edge $\{ \pi_\frp(\psi_1), \pi_\frp(\psi_2)\}$.
Then we can take an element  $s \in G_K$ such that $\pi_\frp(^s\psi_1)=\pi_\frp(\psi_2)$ and 
$\pi_\frp(^s\psi_2)=\pi_\frp(\psi_1)$.
Indeed, if not, then all of $\pi_\frp\langle\phi\rangle$ lie on one side of the edge
 $\{\pi_\frp(\psi_1), \pi_\frp(\psi_2) \}$, which is impossible.
Fix   such $s\in G_K$.
To prove the lemma, it suffices to show  that $\Ip$ has no $G_K$-fixed vertices and no $G_K$-fixed edges distinct from
$\{\pi_\frp(\psi_1), \pi_\frp(\psi_2) \}$.
If $\Ip$ has a  $G_K$-fixed vertex $\pi_\frp(\psi)$, then  
\[
\delta_\frp \left(   \pi_\frp(\psi),  \pi_\frp(\psi_1) \right)
=\delta_\frp \left(   \pi_\frp({}^s\psi),  \pi_\frp({}^s\psi_1) \right)
=\delta_\frp \left(   \pi_\frp(\psi),  \pi_\frp(\psi_2) \right).
\]
It means  that $\pi_\frp(\psi)$ has the same distance from $\pi_\frp(\psi_1)$ and $\pi_\frp(\psi_2)$;
however, this is impossible.
By the similar observation, we also see that any $G_K$-fixed edge  of $\Ip$ coincides with $\{ \pi_\frp(\psi_1), \pi_\frp(\psi_2)\}$.
Hence it is the center of $\cT(\pi_\frp\langle\phi'\rangle)$ for all $\phi' \in \I$.
\end{proof}

Thus we readily give the proof of Theorem \ref{thmmain}.

\begin{proof}[{Proof of Theorem \ref{thmmain}}]
Since the vertex set  of $\T$ is the singleton $\{\phi \}$ for almost all $\frp$, there are only finitely many prime ideals $\frp_1, \frp_2,\ldots, \frp_n$ such that 
the centers of $\cT(\pi_{\frp_i}\langle\phi\rangle)$ are  edges.
Set 
$\frn:=\prod_{i=1}^n\frp_i$, which depends only on the isogeny class of $\phi$ by Lemma \ref{lemcenter}.
For each $1 \leq i \leq n$, let $\{\psi_{\frp_i}, \psi_{\frp_i}'\}$ be the center of $\cT(\pi_{\frp_i}\langle\phi\rangle)$, so that there is  a $\frp_i$-isogeny 
$
\psi_{\frp_i} \ra \psi_{\frp_i}'$.
Using the bijection {\rm(\ref{eqglue})},  
we can take  two Drinfeld $A$-modules  $\psi,  \psi'\in \I$ such that 
\[
\pi_\frp(\psi)=\left\{
\begin{array}{lcc}
\psi_\frp & \mbox{if} & \frp \mid \frn \\
\mbox{the center of}\ \T &  \mbox{if} &  \frp \nmid \frn
\end{array}
\right.
\]
and 
\[
\pi_\frp(\psi')=\left\{
\begin{array}{lcc}
\psi'_\frp & \mbox{if} & \frp \mid \frn \\
\mbox{the center of}\ \T &  \mbox{if} &  \frp \nmid \frn
\end{array}
\right.
\]
for any $\frp$.
By construction,  it follows that $\delta_\frp(\psi, \psi') =1$ if  $\frp \mid \frn$, and that $\delta_\frp(\psi, \psi') =0$ if  $\frp \nmid \frn$.
Hence  there exists a cyclic  $\frn$-isogeny $\psi \ra \psi'$ and  we obtain 
 a $\sepK$-rational point $x:=[\psi\ra\psi'] \in Y_0(\frn)(\sepK)$.

Let $\cP:=\cW(\frn)x$ be the $\cW(\frn)$-orbit of $x$.
If  $\cP$ is $G_K$-stable, then it  gives raise to a $K$-rational point of $Y_*(\frn)$ and hence  (i) holds. 
Therefore it suffices to show that for any $s \in G_K$, there exists an ideal  $\frm_s \mid \frn$ such that ${}^sx=w_{\frm_s} x$.
Let $s \in G_K$.
For each  $\frp$ with  $\frp \mid \frn$, since $\{\pi_\frp(\psi), \pi_\frp(\psi') \}$ is the center of $\T$, we see that 
$\pi_\frp(^s\psi)$ is either $\pi_\frp(\psi)$ or $\pi_\frp(\psi') $.
Define $\frm_s $ to be the product of all prime factors $\frp \mid \frn$ satisfying 
$
\pi_\frp(^s\psi)=\psi'_\frp.
$
If $\frp$ satisfies $\frp \nmid \frn$, then  $\pi_\frp(^s\psi)=\pi_\frp(\psi)$ and $\pi_\frp(^s\psi')=\pi_\frp(\psi')$  by construction.
Thus we have
\[
\pi_\frp(^s\psi)=\left\{
\begin{array}{ccc}
\pi_\frp(\psi') & \mbox{if} & \frp \mid \frm_s \\
\pi_\frp(\psi) &  \mbox{if} &  \frp \nmid \frm_s
\end{array}
\right.
\]
and 
\[
\pi_\frp(^s\psi')=\left\{
\begin{array}{ccc}
\pi_\frp(\psi) & \mbox{if} & \frp \mid \frm_s \\
\pi_\frp(\psi') &  \mbox{if} &  \frp \nmid \frm_s
\end{array}
\right.
\]
for any  $\frp$.
Let us decompose $\frn$ as $\frn=\frm_s \frn'$.
Let    $\psi_{\frm_s} \ra \psi'_{\frm_s}$ be  a $\sepK$-rational cyclic $\frn$-isogeny representing
$w_{\frm_s}x$.
As in  the diagram (\ref{diagram}), we have 
\[
\xymatrix{
\psi \ar[rr]^-\frn \ar[rd]_-{\frm_s}& & \psi' \\
 & \psi_{\frm_s} \ar[rr]^-\frn \ar[rd]_-{\frm_s} \ar[ru]^{\frn'}& & \psi'_{\frm_s} \\
 & & \psi \ar[ru]_-{\frn'}& \hspace{10pt},
}
\]
where the notation such as $\overset{\frn}{\lra}$ means a cyclic $\frn$-isogeny.
Then we have 
\[
\delta_\frp(\psi_{\frm_s},{}^s\psi)=\delta_\frp\left(\pi_\frp(\psi_{\frm_s}),\pi_\frp({}^s\psi)\right)=0 
\]
for any $\frp$.
Hence   $\psi_{\frm_s}={}^s\psi$.
Applying similar arguments to dual isogenies of $\psi \ra \psi'$ and $\psi_{\frm_s } \ra \psi'_{\frm_s}$, we also  have $\psi_{\frm_s}'={}^s\psi'$.
Consequently, $\Theta(w_{\frm_s}x)=\Theta({}^sx)$  
and hence  $w_{\frm_s}x={}^sx$ by  Lemma \ref{lemmodular}.

It remains to check (ii).
To do this,
let $\frn'$ be an ideal satisfying (i).
Then there is a point $y_* \in Y_*(\frn')(K)$ such that $\cP(y_*)$ gives raise to a family of Drinfeld $A$-modules isogenous to $\phi$.
Here  
assume that  $\frn \nmid \frn'$ and 
take a prime factor $\frp$ of $\frn$ with $\frp \nmid \frn'$.
This assumption implies that  all Drinfeld $A$-modules arising from $\cP(y_*)$ have the same image under $\pi_\frp$ since they are joined with each other by isogenies of degree prime to $\frp$.
Hence $\cP(y_*)$ determines  a $G_K$-fixed vertex of $\Ip$.
However, since  now the center of $\T$ is an edge, there are no $G_K$-fixed vertices of $\Ip$ by the proof of Lemma \ref{lemcenter}.
\end{proof}

\begin{rem}
For a finite extension $k$ of $\bQ$,
as  a higher-dimensional generalization of $\bQ$-curves,  
 the notion of {\it abelian $k$-varieties} are  studied by many people.
For example, it is known that   abelian $k$-varieties with some conditions (so-called {\it building blocks}) can be  defined up to isogeny over a  polyquadratic extension of $k$; see \cite{Rib94} and \cite{Pyl??}. 
In \cite{GM}, abelian $k$-surfaces  with quaternionic multiplication are parametrized by $k$-rational points of Atkin-Lehner quotients of Shimura curves.
In the function field setting, as  higher-dimensional generalizations of Drinfeld $A$-modules and analogues of abelian varieties, Anderson \cite{Anderson} defined {\it abelian $t$-modules} and the dual notion of {\it $t$-motives}.
For such objects, we can consider   the notion of  ``virtually $K$-rationality'' in the same way.
However, our proof in this paper depends on special properties of rank-two Drinfeld $A$-modules and so 
it may be  difficult to extend  our arguments to  higher-dimensional cases.

\end{rem}


\section*{Acknowledgements}
The author would like to thank Professor Yuichiro Taguchi for his helpful comments and advises on this paper.
The author also would like to thank Professor Takeshi Saito who recommended using the terminology 
``virtually $K$-rational''.
This work is supported by Foundation of Research Fellows, The Mathematical Society of Japan.



\hspace{50pt}

Department of Mathematics, 

Tokyo Institute of Technology

2-12-1 Oh-okayama, Meguro-ku, 

Tokyo 152-8551, Japan

{\it  E-mail address} : \email{\tt okumura.y.ab@m.titech.ac.jp}

\end{document}